\documentclass{amsart}
\usepackage{amsmath,amsthm,amsfonts,amssymb,mathrsfs}
\usepackage[all]{xy}

\RequirePackage[dvipsnames,usenames]{color}  

\usepackage[total={6.5in,8.75in}, top=1.2in, left=1.0in, includefoot]{geometry}
\usepackage{amssymb}
\usepackage{latexsym}
\usepackage{amsmath,amsthm}
\usepackage{aeguill}
\usepackage{amssymb}
\usepackage{mathrsfs}
\usepackage{hyperref}
\usepackage{etoolbox}
\usepackage{enumerate}
\usepackage[dvipsnames]{xcolor}
\usepackage{tikz}[2013/12/13]

\makeatletter
\pretocmd{\chapter}{\addtocontents{toc}{\protect\addvspace{15\p@}}}{}{}
\pretocmd{\section}{\addtocontents{toc}{\protect\addvspace{3\p@}}}{}{}
\makeatother

\makeatletter
\def\@tocline#1#2#3#4#5#6#7{\relax
  \ifnum #1>\c@tocdepth 
  \else
    \par \addpenalty\@secpenalty\addvspace{#2}%
    \begingroup \hyphenpenalty\@M
    \@ifempty{#4}{%
      \@tempdima\csname r@tocindent\number#1\endcsname\relax
    }{%
      \@tempdima#4\relax
    }%
    \parindent\z@ \leftskip#3\relax \advance\leftskip\@tempdima\relax
    \rightskip\@pnumwidth plus4em \parfillskip-\@pnumwidth
    #5\leavevmode\hskip-\@tempdima
      \ifcase #1
       \or\or \hskip .5em \or \hskip 1em \else \hskip 1.5em \fi%
      #6\nobreak\relax
    \dotfill\hbox to\@pnumwidth{\@tocpagenum{#7}}\par
    \nobreak
    \endgroup
  \fi}
\makeatother
\newcommand{\C}{\mathbb{C}}
\newcommand{\N}{\mathbb{N}}
\newcommand{\Z}{\mathbb{Z}}

\newcommand{\Q}{\mathbb{Q}}
\newcommand{\F}{\mathbb{F}}
\newcommand{\A}{\mathbb{A}}

\newcommand{\Gal}{\operatorname{Gal}}

\renewcommand{\ss}{\operatorname{ss}}
\newcommand{\der}{\operatorname{der}}

\newcommand{\End}{\operatorname{End}}

\newcommand{\conn}{\operatorname{conn}}

\newcommand{\GL}{\mathrm{GL}}

\newcommand{\GSp}{\mathrm{GSp}}

\def\bG{\mathbf{G}}

\def\bT{\mathbf{T}}

\newcommand\uC{\underline{C}}

\newcommand\uG{\underline{G}}

\newcommand\uT{\underline{T}}

\newcommand\uZ{\underline{Z}}

\theoremstyle{plain}
\newtheorem{thm}{Theorem}[section]

\newtheorem{prop}[thm]{Proposition}
\newtheorem{lem}[thm]{Lemma}
\newtheorem{cor}[thm]{Corollary}

\newtheorem{remark}[thm]{Remark}

\begin{document}
	
	\title{Comparison of component groups of $\ell$-adic and mod $\ell$ monodromy groups}
	
	\author{Boyi Dai}
	\address{Department of Mathematics, HKU, Pokfulam, Hong Kong}
	\email{daiboy@connect.hku.hk}
	
	\author{Chun-Yin Hui}
\address{Department of
 Mathematics, The University of Hong
 Kong, Pokfulam, Hong Kong}
\email{pslnfq@gmail.com}

\begin{abstract}
Let $\{\rho_\ell:\Gal_K\to \GL_n(\Q_\ell)\}_\ell$
be a semisimple compatible system of $\ell$-adic representations of a number field $K$
that is arising from geometry.
Let $\bG_\ell\subset\GL_{n,\Q_\ell}$ and $\widehat\uG_\ell\subset\GL_{n,\F_\ell}$ 
be respectively the algebraic monodromy group and the
full algebraic envelope of $\rho_\ell$. We prove that there is a natural 
isomorphism between the component groups 
$\pi_0(\bG_\ell)$ and $\pi_0(\widehat\uG_\ell)$
for all sufficiently large $\ell$.
\end{abstract}

	\maketitle
	
	
\section{Introduction}
\subsection{Compatible systems arising from geometry}
	Let $K$ be a number field, $\overline K$ be an algebraic closure of $K$,
	and $\Gal_K$ be the absolute Galois group $\Gal(\overline K/K)$.
	Denote by $\Sigma_K$ the set of finite places of $K$ and 
	by $S_\ell\subset\Sigma_K$ the subset of places of $K$ that divide a rational prime $\ell$.
A family of (continuous) $n$-dimensional $\ell$-adic representations 
\begin{equation*}\label{cs}
\{\rho_\ell:\Gal_K\to \GL_n(\Q_\ell)\}_\ell
\end{equation*}
of $K$ indexed by the set of rational primes $\ell$, is said to be a 
\emph{(Serre) compatible system} \cite[Chap. 1]{Se98},
if there is a finite subset $S\subset\Sigma_K$ and for 
each $v\in \Sigma_K\backslash S$ a polynomial $P_v(T)\in\Q[T]$
such that the following conditions hold for all $\ell$:
\begin{enumerate}[(a)]
\item the representation $\rho_\ell$ is unramified outside $S\cup S_\ell$ and
\item for all $v\in \Sigma_K\backslash (S\cup S_\ell)$ 
the characteristic polynomial of $\rho_\ell(Frob_v)$ is equal to $P_v(T)$, where $Frob_v$ denotes the Frobenius class at $v$.
\end{enumerate}
A basic source of compatible systems is from geometry.
Let $X$ be a smooth projective variety defined over $K$ and $w\in\Z_{\geq0}$. 
For every prime $\ell$, The $\ell$-adic \'etale cohomology group 
$V_\ell:=H^w_{\text{\'et}}(X_{\overline K},\Q_\ell)$ is a $\Gal_K$-representation,
which is conjectured to be semisimple by Grothendieck-Serre (see \cite{Ta65}). 
Deligne \cite{De74} proved that the family $\{V_\ell\}_\ell$ is a compatible system.
In this article, a compatible system 
$\{\rho_\ell\}_\ell$ is said to be \emph{semisimple} if each $\rho_\ell$ is semisimple;
$\{\rho_\ell\}_\ell$ is said to be \emph{arising from geometry} if there exist 
 smooth projective varieties $X_1,...,X_k$ defined over $K$, $w_1,...,w_k\in\Z_{\geq0}$, and $m_1,...,m_k\in\Z$ such that $\rho_\ell$ is a subquotient of 
\begin{equation*}\label{Uell}
U_\ell:=\bigoplus_{1\leq i\leq k} H^{w_i}_{\text{\'et}}(X_{i,\overline K},\Q_\ell(m_i))
\end{equation*}
for all $\ell$, where $\Q_\ell(m_i)$ denotes the $m_i$th Tate twist. 

\subsection{Algebraic monodromy groups and full algebraic envelopes}
Given a compatible system $\{\rho_\ell\}_\ell$, we define $\bar\rho_\ell^{\ss}:\Gal_K\to\GL_n(\F_\ell)$ 
to be the semisimplification of the reduction modulo $\ell$ of $\rho_\ell$.
Denote by $\Gamma_\ell\subset\GL_n(\Q_\ell)$ the image of $\rho_\ell$
(i.e., \emph{$\ell$-adic monodromy group}),
 by $\bG_\ell\subset\GL_{n,\Q_\ell}$ the the Zariski closure of $\Gamma_\ell$ in $\GL_{n,\Q_\ell}$
(i.e., \emph{algebraic monodromy group}), and by
$\bar\Gamma_\ell\subset\GL_n(\F_\ell)$ the image of $\bar\rho_\ell^{\ss}$  
(i.e., \emph{mod $\ell$ monodromy group}).
It is interesting to study the families $\{\Gamma_\ell\}_\ell$, $\{\bG_\ell\}_\ell$, 
and $\{\bar\Gamma_\ell\}_\ell$ of groups. 
The following is obtained by Serre.

\begin{thm}\label{Serre1} \cite{Se81}
Let $\{\rho_\ell:\Gal_K\to\GL_n(\Q_\ell)\}_\ell$ be a semisimple compatible system. 
\begin{enumerate}[(i)]
\item The finite Galois extension $K^{\conn}/K$ corresponding to the morphism 
$$\Gal_K\stackrel{\rho_\ell}{\longrightarrow}\bG_\ell(\Q_\ell)\to \bG(\Q_\ell)/\bG_\ell^\circ(\Q_\ell)$$
is independent of $\ell$. In particular, the component group 
$\pi_0(\bG_\ell):=\bG_\ell/\bG_\ell^\circ$ is isomorphic to $\Gal(K^{\conn}/K)$ for all $\ell$.
\item The formal character of $\bG_\ell\subset\GL_{n,\Q_\ell}$ is independent of $\ell$.
In particular, the rank of $\bG_\ell$ is independent of $\ell$.
\end{enumerate}
\end{thm}

Suppose the compatible system is 
given by $\{V_\ell:=H^w_{\text{\'et}}(X_{\overline K},\Q_\ell)\}_\ell$ for some 
smooth projective variety $X/K$ as above. 
It is conjectured in \cite[$\mathsection\mathsection$ 9-10]{Se94} that 
the algebraic monodromy group $\bG_\ell$ 
is \emph{independent of $\ell$} and the compact $\ell$-adic Lie group 
$\Gamma_\ell$ is \emph{large} in $\bG_\ell(\Q_\ell)$
in the sense that there exist a  constant $C>0$ and a
 reductive group $G$ defined over $\Q$ together with isomorphisms
$i_\ell:\bG_\ell\stackrel{\simeq}{\rightarrow} G\times_\Q\Q_\ell$ for all $\ell$ such that  the index  
$[G(\Z_\ell): i_\ell(\Gamma_\ell)]\leq C$ for all sufficiently large $\ell$.
A well-known result on these conjectures
is the following theorem of Serre\footnote{When $g=1$ (the elliptic curve case), the result is obtained in \cite{Se72}.}. 

\begin{thm}\cite{Se85,Se86}\label{Serre2}
Let $A$ be a $g$-dimensional abelian variety defined over $K$ such that $\End(A_{\overline K})=\Z$ and 
 either $g$ is odd or $g\in\{2,6\}$. Let $\{\rho_\ell:\Gal_K\to\GL_{2g}(\Q_\ell)\}_\ell$
be the compatible system attached to $\ell$-adic Tate modules of $A$. Then the following assertions hold.
\begin{enumerate}[(i)]
\item The algebraic monodromy group $\bG_\ell \simeq \GSp_{2g,\Q_\ell}$ for all $\ell$.
\item The $\ell$-adic monodromy group $\Gamma_\ell$ is an open subgroup in $\GSp_{2g}(\Q_\ell)$ 
for all $\ell$ and is equal to $\GSp_{2g}(\Z_\ell)$ 
for sufficiently large $\ell$.
\end{enumerate}
\end{thm}

Here $\rho_\ell$ is equivalent to the dual representation of $H^1_{\text{\'et}}(A_{\overline K},\Q_\ell)$.
To establish Theorem \ref{Serre2}(ii) (largeness of monodromy), 
Serre first constructed some connected reductive 
subgroup $\uG_\ell\subset\GL_{2g,\F_\ell}$ for all $\ell\gg0$ 
such that the mod $\ell$ monodromy $\bar\Gamma_\ell$ and $\uG_\ell(\F_\ell)$, as subgroups of $\GL_{2g}(\F_\ell)$,
are commensurate (uniformly independent of $\ell$); and then to show that $\uG_\ell=\GSp_{2g,\F_\ell}$ 
(same root datum as $\bG_\ell$) for all $\ell\gg0$. 
These connected $\F_\ell$-reductive subgroups $\uG_\ell\subset\GL_{2g,\F_\ell}$
are called the \emph{algebraic envelopes} of $\{\rho_\ell\}_\ell$ 
and should conjecturally be given as the reduction modulo $\ell$ of $G^\circ/\Q$. 
They are useful algebraic tools for studying the mod $\ell$ monodromy groups $\bar\Gamma_\ell\subset\GL_n(\F_\ell)$.
We remark that in the seminal work \cite{LP11}, Larsen-Pink developed an algebro-geometric approach to study finite subgroups of $\GL_n(\overline\F_\ell)$.

For compatible systems $\{\rho_\ell\}_\ell$ satisfying certain local conditions (e.g., those arising from geometry),
we constructed algebraic envelopes $\uG_\ell\subset\GL_{n,\F_\ell}$  for $\ell\gg0$
with many nice properties  
and obtained some $\ell$-independence results for them \cite{Hu15,Hu23} 
(see Theorem \ref{Hui2} for details). 
Since the algebraic envelopes $\uG_\ell$ are $\F_\ell$-analogues of the identity components of 
the algebraic monodromy groups $\bG_\ell^\circ$, it is natural to define 
the \emph{full algebraic envelope} to be
\begin{equation}\label{fullae}
\widehat{\uG}_\ell:= \bar\Gamma_\ell\cdot\uG_\ell\subset\GL_{n,\F_\ell}
\end{equation}
(analogous to $\bG_\ell\subset\GL_{n,\Q_\ell}$) for $\ell\gg0$.
The closed subgroup $\widehat{\uG}_\ell\subset \GL_{n,\F_\ell}$ is well-defined for $\ell\gg0$ 
as $\bar\Gamma_\ell$ normalizes $\uG_\ell$, see the proof of \cite[Proposition 3.14]{Hu23}.
It follows that the identity component of $\widehat{\uG}_\ell$ is $\uG_\ell$
and the index $[\widehat{\uG}_\ell(\F_\ell):\bar\Gamma_\ell]$ is bounded by a constant independent of $\ell$.

\subsection{Main results}
The goal of this article is to prove an analogue of Theorem \ref{Serre1}(i) for the component groups of
the full algebraic envelopes $\widehat{\uG}_\ell$ defined in \eqref{fullae}.

\begin{thm}\label{main}
Let $\{\rho_\ell:\Gal_K\to\GL_n(\Q_\ell)\}_\ell$ be a semisimple compatible system
that is arising from geometry, with algebraic monodromy groups $\{\bG_\ell\}_\ell$ 
and full algebraic envelopes $\{\widehat{\uG}_\ell\}_{\ell\gg0}$.
Let $K^{\conn}/K$ be the finite Galois extension corresponding to $\bG_\ell/\bG_\ell^\circ$ 
which is independent of $\ell$.
For all sufficiently large $\ell$, the finite Galois extension corresponding to the morphism 
$$\Gal_K\stackrel{\bar\rho_\ell^{\ss}}{\longrightarrow}\widehat{\uG}_\ell(\F_\ell)\to 
\widehat{\uG}_\ell(\F_\ell)/\uG_\ell(\F_\ell)$$
is $K^{\conn}/K$. In particular, the component groups 
$\pi_0(\bG_\ell)=\bG_\ell/\bG_\ell^\circ$ and $\pi_0(\widehat{\uG}_\ell)=\widehat{\uG}_\ell/\uG_\ell$ 
are naturally isomorphic for all $\ell\gg0$.
\end{thm}

As an immediate consequence, 
the size of the subset of the mod $\ell$ monodromy $\bar\Gamma_\ell$ lying 
outside $\uG_\ell$ can be determined for $\ell\gg0$. 

\begin{cor}\label{maincor}
Let $\{\rho_\ell\}_\ell$ be the semisimple compatible system in Theorem \ref{main}.
For all sufficiently large $\ell$, the mod $\ell$ monodromy $\bar\Gamma_\ell$ satisfies
$$[\bar\Gamma_\ell: \bar\Gamma_\ell\cap \uG_\ell(\F_\ell)]=[K^{\conn}:K].$$
\end{cor}

\subsection{Connections to recent works}
Let $A$ be an abelian variety defined over a number field (or even a finitely generated field over $\Q$).
Since the works of Serre \cite{Se85,Se86},
the method of algebraic envelopes has been an important tool to study the Galois monodromy groups
of $A/K$, see \cite{Wi02},\cite{HL20},\cite{Zy22a,Zy22b} for $\ell$-adic monodromy when $\ell\gg0$
and \cite{Wa14},\cite{Ca15},\cite{LSTX19},\cite{Zy23} for adelic monodromy under specializations. 
This method is also used to study the Geyer-Jarden conjecture \cite{GJ78}
(resp. a conjecture of Hindry-Ratazzi \cite{HR12}), which concerns
the arithmetic of the torsion points of $A$ over large extensions (resp. finite extensions) of $K$,
see the works \cite{AGP13a,AGP13b},\cite{Zy16},\cite{JP19} 
(resp. \cite{HR10,HR12,HR16},\cite{LLZ23}).

As the (connected) algebraic envelopes $\uG_\ell$ only capture the mod $\ell$ monodromy groups  after a finite extension of $K$,
our main results on the full algebraic envelopes $\widehat{\uG}_\ell$
shall give a more refined study
of the mod $\ell$ monodromy groups of semisimple compatible systems arising from geometry.

\subsection{Outline of the article}
In section $2$, we describe some $\ell$-independence results 
for algebraic monodromy groups $\bG_\ell$ and algebraic envelopes $\uG_\ell$ in \cite{Hu13,Hu15,Hu23}
and prove a useful finiteness result for algebraic envelopes (Proposition \ref{finite}) that is of independent interest. 
Then we state a result (Proposition \ref{contain}) that allows us to reduce Theorem \ref{main}
to the disconnected case (Proposition \ref{mainprop}) in section $3$.
Proposition \ref{mainprop} will be proved  by combining the results in section $2$,
some ideas of Serre \cite{Se81}, and Lang-Weil \cite{LW54} to estimate the numbers of rational points
of some $\F_\ell$-subvarieties for all sufficiently large $\ell$.

\section{$\ell$-independence of algebraic monodromy and algebraic envelopes}
\subsection{Formal bi-characters for reductive groups}
Let $G$ be a reductive group defined over a field $F$.
Denote by $G^\circ$ the identity component of $G$ and 
by $G^{\der}$ the derived group $[G^\circ,G^\circ]$.
If $F'$ is a field extension of $F$, define $G_{F'}:=G\times_F F'$ the base change.

Let $G\subset\GL_{n,F}$ be a reductive subgroup. Suppose first $F$ is algebraically closed.
The \emph{formal character} of $G$ is the $\GL_{n,F}$-conjugacy class
of a maximal torus $T$ of $G$. The \emph{formal bi-character} of $G$ is the
$\GL_{n,F}$-conjugacy class of the chain $T'\subset T$ of subtori where 
$T$ is a maximal torus of $G$ and $T'$ is a maximal torus of $G^{\der}$.
For general $F$ with $\overline F$ an algebraic closure, 
the formal character and formal bi-character of $G$
are defined to be those of $G_{\overline F}$.

Let $\{F_i\}$ be a family of fields and $\{G_i\subset\GL_{n_i,F_i}\}$ be a family of reductive subgroups,
index by a set $I$.
We say that the formal characters of $\{G_i\}$ are \emph{the same} 
if $n_i=n\in\N$ for all $i\in I$ and there is a split $\Z$-subtorus
$T_\Z\subset\GL_{n,\Z}$ such that $T_\Z\times \overline F_i$
is the formal character of $G_i$ for all $i$. This gives an equivalence relation on
the formal characters of reductive subgroups of general linear groups defined over fields.
The formal characters of the family $\{G_i\}$ are said to be \emph{bounded} if
they belong to finitely many classes under the above equivalence.
One defines similarly for when the formal bi-characters of $\{G_i\}$ are the same or bounded.
	
\subsection{$\ell$-independence of algebraic monodromy groups}
	Let $K$ be a number field and $\{\rho_\ell:\Gal_K\to\GL_n(\Q_\ell)\}_\ell$ be a semisimple compatible system.
The algebraic monodromy group $\bG_\ell$ of $\rho_\ell$ is a reductive subgroup of $\GL_{n,\Q_\ell}$
for all $\ell$. Theorem \ref{Serre1}(ii) is generalized to the following.

\begin{thm}\cite{Hu13}\label{Hui1}
The formal bi-character of $\bG_\ell\subset\GL_{n,\Q_\ell}$ is independent of $\ell$.
In particular, the rank (resp. semisimple rank) of $\bG_\ell$ is independent of $\ell$.
\end{thm}
	
\subsection{$\ell$-independence of algebraic envelopes}
In \cite[$\mathsection\mathsection2.8,3.1$]{Hu23},
we constructed algebraic envelopes $\uG_\ell\subset\GL_{n,\F_\ell}$ with nice properties 
for semisimple compatible systems satisfying 
certain local conditions, including those arising from geometry.
Denote by $\bar\epsilon_\ell:\Gal_K\to\F_\ell^\times$ the mod $\ell$ cyclotomic character for all $\ell$. 

\begin{thm}\cite[Theorems 2.11 and 3.1]{Hu23}\label{Hui2}
		Let $\{\rho_{\ell}:\Gal_K\to\GL_n(\Q_\ell)\}_\ell$ be a semisimple compatible system of a number field $K$. 
		Suppose there exist integers $N_1, N_2\ge0$ and a finite extension $K'/K$ such that the 
		following conditions hold.
		\begin{enumerate}[(a)]
\item (Bounded tame inertia weights): for all $\ell\gg0$ and each finite place $v$ of $K$ above $\ell$, 
the tame inertia weights of the local representation $(\bar\rho_\ell^{\ss}\otimes\bar\epsilon_\ell^{N_1})|_{\Gal_{K_v}}$ belong to $[0,N_2]$.
\item (Potential semistability): for all $\ell\gg0$ and each finite place $w$ of $K'$ not above $\ell$,
the semisimplification of the local representation $\bar\rho_\ell^{\ss}|_{\Gal_{K_{w}'}}$ is unramified.
\end{enumerate}
		Then there exist a finite Galois extension $L/K$ (with $K^{\conn}\subset L$) and for each $\ell\gg0$, 
a connected reductive subgroup $\uG_\ell\subset\GL_{n,\F_\ell}$ 
with properties below.
\begin{enumerate}[(i)]
\item The derived group $\uG_\ell^{\der}$ is the Nori group  (\cite{No87})
of $\bar\rho_\ell^{\ss}(\Gal_K)\subset\GL_n(\F_\ell)$.
\item The image $\bar\rho_\ell^{\ss}(\Gal_L)$ is a subgroup of $\uG_\ell(\F_\ell)$ 
with index bounded by a constant independent of $\ell$. 
\item The action of $\uG_\ell$ on the ambient space is semisimple.
\item The formal characters of $\uG_\ell\subset\GL_{n,\F_\ell}$ for all $\ell\gg0$ are bounded.
\item The formal bi-characters of $\bG_\ell$ and $\uG_\ell$ are the same and independent of $\ell$.
\item The commutants of $\bar\rho_\ell^{\ss}(\Gal_L)$ and $\uG_\ell$ (resp. 
$[\bar\rho_\ell^{\ss}(\Gal_L),\bar\rho_\ell^{\ss}(\Gal_L)]$ and $\uG_\ell^{\der}$) in $\End(\F_\ell^n)$ are equal.
\end{enumerate}
The group $\uG_\ell$ is called the algebraic envelope of $\rho_\ell$ 
and is uniquely determined by properties (ii)--(iv) when $\ell$ 
is sufficiently large.
\end{thm}

Note that condition (v) is strictly stronger than (iv) but
(iv) is included for the uniqueness formulation.
Conjecturally, the algebraic monodromy group $\bG_\ell$ 
is independent of $\ell$ and the algebraic envelopes $\uG_\ell$ should 
have the same root datum as $\bG_\ell$ for $\ell\gg0$.
In the following we prove a finiteness result on $\ell$-independence of the algebraic envelopes,
which is crucial to the proof of Theorem \ref{main} later.
Properties (iii) and (v)  of the algebraic envelopes $\uG_\ell$ in Theorem \ref{Hui2} will be used.

\begin{prop}\label{finite}
Let $\{\rho_\ell\}_\ell$ be the semisimple compatible system in Theorem \ref{Hui2}
and $\{\uG_\ell\subset\GL_{n,\F_\ell}\}$ be the algebraic envelopes for all sufficiently large $\ell$.
Then there exist finitely many connected split reductive subgroups $G_1,G_2,...,G_m\subset\GL_{n,\Z[1/N]}$
defined over $\Z[1/N]$ (for some $N\in\N$) such that for each sufficiently large $\ell$,
the base change of the algebraic envelope $\uG_{\ell,\overline\F_\ell}$ is conjugate to $G_{i,\overline\F_\ell}$
in $\GL_{n,\overline\F_\ell}$ for some $i$.
\end{prop}

\begin{proof}
We first treat the semisimple part $\uG_{\ell}^{\der}$ of the algebraic envelope.
The base change $\uG_{\ell,\overline\F_\ell}^{\der}$ (to an algebraic closure)
admits a $\Z$-model $S$ (Chevalley group defined over $\Z$), i.e., an isomorphism
$\iota_\ell: S_{\overline\F_\ell}\stackrel{\simeq}{\longrightarrow} \uG_{\ell,\overline\F_\ell}^{\der}$.
Since the rank of $\uG_{\ell}^{\der}$ is bounded by $n$, such $\Z$-models 
have finitely many possibilities $S_1,S_2,...,S_h$.
For $1\leq i\leq h$, fix $T_i\subset S_i$ a split maximal torus 
and $B_i\subset S_i$ a Borel subgroup containing $T_i$.
For each $i$, we would like to show that there are finitely many $\Z$-representations 
$\{\theta_{ij}:S_i\to\GL_{n,\Z}\}_{j\in J_i}$ such that if $S_i$ is the $\Z$-model 
of $\uG_{\ell,\overline\F_\ell}^{\der}$
and $\ell$ is sufficiently large, then 
\begin{equation}\label{compose}
 S_{i,\overline\F_\ell}\stackrel{\iota_\ell}{\longrightarrow}\uG_{\ell,\overline\F_\ell}^{\der}\to \GL_{n,\overline\F_\ell}
\end{equation}
can be descended to some $\theta_{ij}$. 

Without loss of generality, we assume $h=1$ and write $T\subset B\subset S$ instead of $T_i\subset B_i\subset S_i$.
Let $\mathbb{X}(T):=\mathrm{Hom}(T,\mathbb{G}_{m,\Z})$ be the character group of $T$.
For each highest weight $\lambda\in\mathbb{X}(T)$ (with respect to the Borel $B$),
let  $\theta_\lambda$ be a $\Z$-representation of $S$ that after base change to $\C$,
is irreducible with highest weight $\lambda$ (with respect to $B_\C$).
We need the following result by Springer.

\begin{lem}\label{irred}\cite[Corollary 4.3]{Sp68}
The base change $\theta_\lambda\times\overline\F_\ell$
is an irreducible representation of $S_{\overline\F_\ell}$ for $\ell\gg0$.
\end{lem}

Let $V_\ell$ be the faithful representation $\uG_{\ell,\overline\F_\ell}^{\der}\to\GL_{n,\overline\F_\ell}$ 
and  $\Sigma_\ell\subset\mathbb{X}(\iota_\ell(T_{\overline\F_\ell}))$
be the multiset of weights of $V_\ell$.
Denote by $\iota_\ell^*: \mathbb{X}(\iota_\ell(T_{\overline\F_\ell}))\to\mathbb{X}(T)$ the 
isomorphism of character groups induced by $\iota_\ell$.

\begin{lem}\label{union}
The subset $\bigcup_{\ell\gg0} \iota_\ell^*(\Sigma_\ell)\subset\mathbb{X}(T)$ is finite.
\end{lem}

\begin{proof}
Since  $V_\ell\otimes V_\ell^\vee$ contains the adjoint representation 
of $\uG_{\ell,\overline\F_\ell}^{\der}$ as subrepresentation,
the difference 
$$\Sigma_\ell -\Sigma_\ell:=\{\lambda-\lambda':~(\lambda,\lambda')\in 
\Sigma_\ell^2\}\subset\mathbb{X}(\iota_\ell(T_{\overline\F_\ell}))$$ 
contains the set of roots $R_\ell$ of  $\uG_{\ell,\overline\F_\ell}^{\der}$ with respect to 
the maximal torus $\iota_\ell(T_{\overline\F_\ell})$.
Since the formal character of $\uG_{\ell,\overline\F_\ell}^{\der}$ is independent of $\ell\gg0$ 
by Theorem \ref{Hui2}(v), there exist a finite multiset $\Sigma$ in $\Z^r$ 
and an isomorphism $\mathbb{X}(\iota_\ell(T_{\overline\F_\ell}))\simeq \Z^r$ for $\ell\gg0$
such that $\Sigma_\ell$ and $\Sigma$ correspond. 
Hence for $\ell\gg0$, the multisets $\Sigma_\ell \sqcup (\Sigma_\ell -\Sigma_\ell)$ and 
\begin{equation}\label{weights}
\Sigma \sqcup (\Sigma-\Sigma)
\end{equation}
also correspond.
To prove the lemma, it suffices to show that the image of \eqref{weights} under the map
\begin{equation*}\label{phiell}
\phi_\ell:\Z^r\stackrel{\simeq }{\rightarrow} \mathbb{X}(\iota_\ell(T_{\overline\F_\ell}))\stackrel{\iota_\ell^*}{\rightarrow}\mathbb{X}(T)
\end{equation*}
has finitely many possibilities for all $\ell\gg0$.

Let $R$ be the subset of \eqref{weights} corresponding to $R_\ell$ (the roots).
On the one hand, the map $\phi_\ell$ is determined by its restriction to $R$.
On the other hand, $R$ has finitely many possibilities for $\ell\gg0$ because \eqref{weights} has size $n+n^2$.
Since $\phi_\ell(R)$ must be the set of roots of $S$ by 
the isomorphism $\iota_\ell$, the finiteness assertion holds.
\end{proof}

Let $V_\ell=W_1\oplus W_2\oplus\cdots\oplus W_s$
be the irreducible decomposition of the semisimple representation $V_\ell$ (Theorem \ref{Hui2}(iii))
of $\uG_{\ell,\overline\F_\ell}^{\der}$ for $\ell\gg0$. 
For $1\leq k\leq s$, the irreducible factor $W_k$ corresponds 
to a highest weight 
$$\lambda_k\in \bigcup_{\ell\gg0} \iota_\ell^*(\Sigma_\ell).$$ 
On the one hand, Lemma \ref{union} implies that the $\Z$-representation 
 $\theta_\ell:=\oplus_{k=1}^{s} \theta_{\lambda_k}$ of $S$ 
has only finitely many possibilities for all $\ell\gg0$.
On the other hand, since an irreducible representation 
of $S_{\overline\F_\ell}$ is determined by the highest weight,
it follows from Lemma \ref{irred}
that $S_{\overline\F_\ell}\stackrel{\iota_\ell}{\longrightarrow}\uG_{\ell,\overline\F_\ell}^{\der}\to \GL_{n,\overline\F_\ell}$ can be descended to $\theta_\ell$ for $\ell\gg0$. 
We conclude that there is a finite set of $\Z$-representations 
$$\{\theta_{ij}:S_i\to\GL_{n,\Z}:~ 1\leq i\leq h,~j\in J_i\}$$
such that for $\ell\gg0$, \eqref{compose} can be descended to some $\theta_{ij}$.

Since the image of $\theta_{ij}\times\Q$ (the base change to $\Q$) 
is a connected split semisimple subgroup of $\GL_{n,\Q}$, 
it can be extended (via generic smoothness) to a $\Z[1/N_{ij}]$-smooth closed subgroup subscheme  
$$G_{ij}^{\der}\subset\GL_{n,\Z[1/N_{ij}]}$$
for some $N_{ij}\in\N$. By \cite[Propositions 3.1.9 and 3.1.12]{Co14}, we may assume $G_{ij}^{\der}$
is a reductive group scheme defined over $\Z[1/N_{ij}]$.
Enlarging $N_{ij}$ if necessary, $\theta_{ij}\times \Z[1/N_{ij}]$
factors through  $G_{ij}^{\der}$ such that the image of $\theta_{ij}\times\F_\ell$
is equal to the special fiber $G_{ij,\F_\ell}^{\der}$ for $\ell\gg0$.
Therefore, there exist $N\in\N$ and a finite set of split
reductive subgroup subschemes
\begin{equation}\label{Chevalley}
\{G_i^{\der}\subset\GL_{n,\Z[1/N]}:~1\leq i\leq r\}
\end{equation}
 such that for all $\ell\gg0$, the semisimple group
$\uG_{\ell,\overline\F_\ell}^{\der}$ and some $G_{i,\overline\F_\ell}^{\der}$ 
are conjugate in $\GL_{n,\overline\F_\ell}$.

To finish the proof, for simplicity we assume $r=1$ and (enlarging $N$ if necessary)
$$G^{\der}\subset\prod_{1\leq j\leq s}\GL_{n_j,\Z[1/N]}$$
in \eqref{Chevalley} such that for all $j$,  the $j$th  representation 
$G^{\der}_{\overline\F_\ell}\to\GL_{n_j,\overline\F_\ell}$ after base change
is irreducible for $\ell\gg0$. Assume also $G^{\der}_{\overline\F_\ell}=\uG^{\der}_{\ell,\overline\F_\ell}$
holds and let $\uC_\ell$ be the connected component of the center of $\uG_{\ell,\overline\F_\ell}$.
It follows that 
$$\uC_\ell\subset\mathbb{G}_{m,\overline\F_\ell}^s\subset\prod_{1\leq j\leq s}\GL_{n_j,\overline\F_\ell},$$
where the middle group is the center of the big one.
Consider the morphism
$$\text{Det}:\prod_{1\leq j\leq s}\GL_{n_j,\overline\F_\ell}
\stackrel{}{\longrightarrow}\mathbb{G}_{m,\overline\F_\ell}^s$$
that sends $(A_1,...,A_s)$ to $(\det(A_1),...,\det(A_s))$ (coordinates of determinants).
Let $\uT_\ell$ be a maximal torus of $\uG_{\ell,\overline\F_\ell}$. By construction,
we have 
$$\text{Det}(\uC_\ell)=\text{Det}(\uG_{\ell,\overline\F_\ell})=\text{Det}(\uT_\ell).$$
Since the formal character (i.e., $\uT_\ell$) of $\uG_{\ell,\overline\F_\ell}$ is independent of $\ell\gg0$ (Theorem \ref{Hui2}(v)),
the family of subtori 
$$\{\text{Det}(\uC_\ell)\subset\mathbb{G}_{m,\overline\F_\ell}^s\}_{\ell\gg0}$$
can be given by finitely many $\Z$-subtori of $\mathbb{G}_{m,\Z}^s$. 
Since the restriction of $\text{Det}$
to $\mathbb{G}_{m,\overline\F_\ell}^s$ is a finite morphism onto itself and $\uC_\ell$ is connected,
it follows that the family $\{\uC_\ell\subset\mathbb{G}_{m,\overline\F_\ell}^s\}_{\ell\gg0}$
can also be given by finitely many $\Z$-subtori of $\mathbb{G}_{m,\Z}^s$. 
By adding each of these (finitely many) $\Z$-subtori to $G^{\der}$, we obtain the desired finiteness assertion.
\end{proof}

\begin{remark}
Using similar strategy, one can prove that there exist finitely many connected split reductive subgroups $G_1,G_2,...,G_k\subset\GL_{n,\Q}$
defined over $\Q$ such that for each $\ell$,
the base change of the identity component of algebraic monodromy group $\bG_{\ell,\overline\Q_\ell}^\circ$ 
is conjugate to $G_{i,\overline\Q_\ell}$
in $\GL_{n,\overline\Q_\ell}$ for some $i$.
\end{remark}

\subsection{MFT hypothesis and mod $\ell$ monodromy groups}
Let $\{\rho_\ell:\Gal_K\to\GL_n(\Q_\ell)\}_\ell$ be a semisimple compatible system.
Consider a member $\rho_\ell$ of the system and let $\bG_\ell$ be
its algebraic monodromy group.
If  $\bar v$ is a place of $\overline K$ extending $v\in\Sigma_K\backslash (S\cup S_\ell)$,
then $\rho_\ell$ is unramified at $v$ and 
the image of Frobenius $\rho_\ell(Frob_{\bar v})\in\bG_\ell(\Q_\ell)$ is well-defined.
The \emph{Frobenius torus} at $\bar v$ is defined (by Serre) as 
the identity component $\bT_{\bar v}$ of the smallest algebraic group 
containing the semisimple part of $\rho_\ell(Frob_{\bar v})$; $\bT_{\bar v}$
is a $\Q_\ell$-subtorus of $\bG_\ell$.
 
We say that $\rho_\ell$ satisfies the \emph{maximal Frobenius tori hypothesis} (MFT)
if $\bG_\ell$ is connected and there is a Dirichlet density one subset 
$\mathscr{S}_K\subset\Sigma_K\backslash (S\cup S_\ell)$ such that 
if $\bar v$ is a place of $\overline K$ extending $v\in\mathscr S_K$,
then $\bT_{\bar v}$ is a maximal torus of $\bG_\ell$. 
By Theorem \ref{Serre1} and the compatibility conditions, if one $\rho_\ell$ satisfies MFT
then all $\rho_\ell$ satisfy MFT. We say that $\{\rho_\ell\}_\ell$ satisfies MFT
if some $\rho_\ell$ (hence all) satisfies MFT.
If $\{\rho_\ell\}_\ell$ is arising from geometry and $\bG_\ell$ is connected,
then $\{\rho_\ell\}_\ell$ satisfies MFT 
(the idea goes back to Serre, see \cite[$\mathsection3.5$]{Hu23} and \cite[Theorem 2.6]{Hu18}
for details). 

\begin{prop}\label{contain}\cite[Proposition 3.14]{Hu23}
Let $\{\rho_\ell:\Gal_K\to\GL_n(\Q_\ell)\}_\ell$ be a semisimple compatible system satisfying 
the conditions in Theorem \ref{Hui2} with algebraic envelopes $\uG_\ell\subset\GL_{n,\F_\ell}$.
If the compatible system satisfies MFT, then 
the mod $\ell$ monodromy group 
$$\bar\Gamma_\ell:=\bar\rho_\ell^{\ss}(\Gal_K)\subset\uG_\ell(\F_\ell)$$
for all sufficiently large $\ell$.
\end{prop}

\begin{remark}
This result is first proved in \cite[Theorem 4.5]{HL20}
for the compatible system $\{H^w_{\text{\'et}}(X_{\overline K},\Q_\ell)\}_\ell$ where $X/K$
is a smooth projective variety and the idea of proof can be carried over.
\end{remark}

\section{Comparison of component groups}
	\subsection{Proof of Theorem \ref{main}}
	
Let	$\{\rho_{\ell}:\Gal_K\to\GL_n(\Q_\ell)\}_\ell$ be semisimple compatible system of $\ell$-adic representations that 
is arising from geometry. The semisimplified reduction modulo $\ell$ of 
$\{\rho_\ell\}_\ell$ is $\{\bar\rho_{\ell}^{\ss}:\Gal_K\to\GL_n(\F_\ell)\}_\ell$.
Let $K^{\conn}/K$ be the Galois extension in Theorem \ref{Serre1}(i).
 We record the following groups for every $\ell$.
\begin{itemize}
\item The $\ell$-adic monodromy group $\Gamma_\ell:=\rho_\ell(\Gal_K)\subset\GL_n(\Q_\ell)$;
\item The mod $\ell$ monodromy group $\bar\Gamma_\ell:=\bar\rho_{\ell}^{\ss}(\Gal_K)\subset\GL_n(\F_\ell)$;
\item The algebraic monodromy group $\bG_\ell\subset\GL_{n,\Q_\ell}$.
	\item $\Gamma_\ell^\circ:=\rho_\ell(\Gal_{K^{\conn}})$;
	\item $\bar\Gamma_\ell^\circ:=\bar\rho_\ell^{\ss}(\Gal_{K^{\conn}})$.
\end{itemize}

Since $\{\rho_\ell\}_\ell$ satisfies the conditions in Theorem \ref{Hui2} ($\mathsection2.3$), 
algebraic envelopes $\uG_\ell\subset\GL_{n,\F_\ell}$ can be attached for $\ell\gg0$.
By Theorem \ref{Hui2}(ii), there is a finite Galois extension $L/K$ such 
that $\bar\rho_\ell^{\ss}(\Gal_L)\subset \uG_\ell(\F_\ell)\subset\GL_n(\F_\ell)$ for $\ell\gg0$.
Since $\bar\Gamma_\ell$ normalizes $\bar\rho_\ell^{\ss}(\Gal_L)$, the uniqueness assertion
in Theorem \ref{Hui2} implies that $\bar\Gamma_\ell$ normalizes $\uG_\ell$ for $\ell\gg0$.
Hence, we have following groups for $\ell\gg0$.

\begin{itemize}
	\item The algebraic envelope $\uG_\ell\subset\GL_{n,\F_\ell}$;
	\item The full algebraic envelope
	$\widehat{\uG}_\ell:= \bar\Gamma_\ell\cdot\uG_\ell\subset\GL_{n,\F_\ell}$.
\end{itemize}

Note that the identity component of $\widehat{\uG}_\ell$ is $\uG_\ell$.
We deduce a natural morphism between the component groups.

\begin{prop}\label{componentmap}
Let	$\{\rho_{\ell}:\Gal_K\to\GL_n(\Q_\ell)\}_\ell$ be a semisimple compatible system
that is arising from geometry. For all sufficiently large $\ell$, we 
have a natural surjection
\begin{equation}\label{natural}
\pi_0(\bG_\ell):=\bG_\ell/\bG_\ell^\circ\simeq \Gamma_\ell/\Gamma_\ell^\circ\twoheadrightarrow
\bar\Gamma_\ell/\bar\Gamma_\ell^\circ\twoheadrightarrow \widehat{\uG}_\ell/\uG_\ell=:\pi_0(\widehat{\uG}_\ell).
\end{equation}
\end{prop}

\begin{proof}
It suffices to show that $\bar\Gamma_\ell^\circ\subset\uG_\ell(\F_\ell)$ for $\ell\gg0$.
Since the restriction $\{\rho_\ell|_{\Gal_{K^{\conn}}}\}_\ell$ is also a semisimple compatible system 
arising from geometry, the uniqueness assertion in Theorem \ref{Hui2}
implies that the algebraic envelopes of the restriction system are likewise $\uG_\ell$ for $\ell\gg0$.
Since the restriction system satisfies MFT (see $\mathsection2.5$), we are done by Proposition \ref{contain}.
\end{proof}
	
To prove Theorem \ref{main}, we have to prove that \eqref{natural} is an isomorphism for $\ell\gg0$.
Since $\pi_0(\bG_\ell)\simeq \Gal(K^{\conn}/K)$ for all $\ell$, we obtain
a Galois subextension $K_\ell$ of $K^{\conn}/K$ such that $\pi_0(\widehat{\uG}_\ell)\simeq \Gal(K_\ell/K)$
for $\ell\gg0$. If Theorem \ref{main} is false, then there exist  an infinite set $\mathcal L$ 
of rational primes and a Galois subextension $K'$ of $K^{\conn}/K$ such that 
$K'\subsetneq K^{\conn}$ and $K_\ell=K'$ for all $\ell\in\mathcal L$.
Since the restriction system $\{\rho_\ell|_{\Gal_{K'}}\}_\ell$ also 
has $\{\uG_\ell\subset\GL_{n,\F_\ell}\}_{\ell\gg0}$ as algebraic envelopes (by the uniqueness assertion 
of Theorem \ref{Hui2}), 
we may assume $\bG_\ell$ is not connected (for all $\ell$) and 
the full algebraic envelope
$\widehat{\uG}_\ell=\uG_\ell$ is connected for $\ell\in\mathcal L$ by replacing $K$ with $K'$.
This is impossible by the following proposition and thus Theorem \ref{main} holds.

\begin{prop}\label{mainprop}
Let	$\{\rho_{\ell}:\Gal_K\to\GL_n(\Q_\ell)\}_\ell$ be a semisimple compatible system
that is arising from geometry. If the algebraic monodromy $\bG_\ell$ is not connected for some $\ell$, 
then the full algebraic envelope $\widehat{\uG}_\ell$ is not connected for all 
sufficiently large $\ell$.
\end{prop}

\subsection{Proof of Proposition \ref{mainprop}}
	Let $F$ be a field of characteristic zero. 
	For a matrix $g\in\mathrm{GL}_n(F)$, 
	write $\det(TI_n-g)=T^n+\sum_{i=1}^n\alpha_i(g)T^{n-i}$  as the characteristic polynomial of $g$.
	Define the $F$-morphism
	$$\chi:\mathrm{GL}_{n,F}\to\mathbb{A}^{n-1}_F\times\mathbb{G}_{m,F}$$
	that sends $g\in\GL_n(F)$ to the coefficients $(\alpha_1(g),...,\alpha_n(g))$. 
	We present two results of Serre \cite{Se81} 
	that study the components of an algebraic subgroup $G\subset\GL_{n,F}$ 
	(e.g., $\bG_\ell\subset\GL_{n,\Q_\ell}$) via this map.
	
	\begin{lem}(Serre)\label{poly}
		Let $G\subseteq\mathrm{GL}_{n,F}$ be an algebraic subgroup
		and $g\in G(F)\backslash G^{\circ}(F)$. 
		There exists a polynomial $f\in\mathbb{Z}[\alpha_1,...,\alpha_n]$ 
		such that $f(\chi(gG^{\circ}))=0$ and $f(\chi(\mathrm{id}))\not=0$.
	\end{lem}
	
\begin{proof}
We present Serre's proof.
Since the unipotent radical of $G$ is connected and $\chi$ is defined using the characteristic polynomial,
we assume $G$ is reductive by semisimplification of $G\subset\GL_{n,F}$. 
		Denote by $\pi_0(G)$ the component group of $G$. Fix a representation $\rho:\pi_0(G)\to\mathrm{GL}_W$ with $\rho(gG^{\circ})\not=1$. Write the same symbol $W$ as the $G$-representation $G\to\pi_0(G)\to\mathrm{GL}_W$ and let $V$ be the faithful representation $G\to\GL_{n,F}$ of $G$. 
	For $a,b\in\Z_{\geq 0}$, denote by $T^{a,b}(V):=(V^{\otimes a})\otimes(V^{\vee,\otimes b})$ the representation of $G$.
		As $G$ is reductive, 	$W$ is a subrepresentation of some $\oplus_{1\le i\le r}T^{a_i,b_i}(V)$ \cite[I. Proposition 3.1(a)]{DMOS82}.
		
		Assume $g$ has generalized eigenvalues $\lambda_1,\lambda_2,...,\lambda_n$ in $V$ counting multiplicity. Then the characteristic polynomial of $g$ in $T^{a,b}(V)$ is:
		\begin{equation}\label{1}
		\prod_{\substack{1\le j_1,...,j_a\le n\\1\le k_1,\cdots,k_b\le n}}\left(T-\frac{\lambda_{j_1}\cdots\lambda_{j_a}}{\lambda_{k_1}\cdots\lambda_{k_b}}\right)
		\end{equation}
		After multiplying \eqref{1} by a high power of $\alpha_n(g)=\prod_{1\le i\le n}\lambda_i$, we obtain an integral polynomial 
		$$P_{a,b}(T,\alpha_1(g),...,\alpha_n(g))\in \Z[T,\alpha_1(g),...,\alpha_n(g)].$$

		As $\pi_0(G)$ is finite and $\rho(g)\not=1$, some eigenvalue of $\rho(g)$ is a $m$th primitive root of unity with $m>1$. By letting $T=\zeta$ run through all the $m$th primitive roots of unity
		and then taking the product, we obtain an integral polynomial:
		$$Q_{a,b}(\alpha_1,...,\alpha_n):=\prod_{\zeta}P_{a,b}(\zeta,\alpha_1,...,\alpha_n)\in\Z[\alpha_1,...,\alpha_n].$$
		
		Finally, we define 
		$$f(\alpha_1,...,\alpha_n):=\prod_{1\le i\le r}Q_{a_i,b_i}(\alpha_1,...,\alpha_n)
		\in \Z[\alpha_1,...,\alpha_n]$$ 
	such that $f(\chi(gG^{\circ}))=0$ by construction. 
	It also follows that $f(\chi(\mathrm{id}))\not=0$ since $\zeta-1\neq 0$.
	\end{proof}
	
The second result below is a consequence of Chebotarev's density theorem.
	
	\begin{prop}\cite[Theorem 3]{Ra98}\label{density}	
		Let $\rho_\ell:\mathrm{Gal}_K\to\mathrm{GL}_n(\mathbb{Q}_{\ell})$ be a semisimple 
		$\ell$-adic Galois representation
		that is unramified outside a finite set of places and let $\bG_\ell\subset\GL_{n,\Q_\ell}$ 
		be the algebraic monodromy group. For $f \in\mathbb{Z}[\alpha_1,...,\alpha_n]$, let $k$ the number of components of $\bG_\ell$ such that $f\circ\chi$ is identically zero. 
		Then the natural density $d(\mathscr S_\ell)$ of 
\begin{equation}\label{S}
\mathscr S_\ell:=\{v\in\Sigma_K:~\rho_\ell~\text{is unramified at }v\text{~and~} 
f\circ \chi(\rho_\ell(\mathrm{Frob}_v))=0\}
\end{equation}
		is $\frac{k}{|\pi_0(\bG_\ell)|}$. 
		In particular, if $\bG_\ell$ is not connected and we take $f$ as in Lemma \ref{poly} ($F=\Q_\ell$), 
		then the natural density $d(\mathscr S_\ell)>0$.
	\end{prop}

 We need an upper bound on the number of rational points from Lang-Weil \cite{LW54}.
		
	\begin{lem}\cite[Lemma 1]{LW54}\label{LW}
		Given integers $n,d\geq 0$ and $r\geq 1$, there exists a positive number $A(n,d,r)$ depending
		only on $n,d,r$ such that for any finite field $\F_q$ and closed (not necessarily irreducible)
		subvariety $X\subset \mathbb{P}_{\F_q}^n$ 
		of degree $d$ and dimension $r$:
		$$|X(\F_q)|\leq A(n,d,r) q^r.$$
	\end{lem}
	
For linear algebraic groups defined over $\F_\ell$, we have good estimates.

\begin{lem}\cite[Lemma 3.5]{No87}\label{gpbound}
Let $A$ be an $r$-dimensional connected linear algebraic group defined over $\F_\ell$. Then
$$(\ell-1)^r\leq |A(\F_\ell)|\leq (\ell+1)^r.$$
\end{lem}

	We are now ready to prove Proposition \ref{mainprop}.
	
	\begin{proof}
		Fix a rational prime $\ell'$.
		Since $\bG_{\ell'}$ is not connected, pick a  place $\bar v$ of $\overline K$ (lying above $v\in\Sigma_K\backslash S$) 
		such that $\rho_{\ell'}$ is unramified 
		and $g:=\rho_{\ell'}(Frob_{\bar v})$ does not belong to $\bG_{\ell'}^\circ$.
		By Lemma \ref{poly}, there is a polynomial $f\in\Z[\alpha_1,...,\alpha_n]$
		such that $f(\chi(g\bG_{\ell'}^\circ))=0$
		and $f(\chi(\mathrm{id}))\neq 0$.
		Let $\mathscr S_\ell\subset\Sigma_K$ be the subset defined in \eqref{S}.
		
		\begin{prop}\label{denindep}
		The natural density $d(\mathscr S_\ell)$ is independent of $\ell$ and is equal to a constant $C_1>0$.
		\end{prop}
		
		\begin{proof}
		By Proposition \ref{density},
		the natural density $d(\mathscr S_{\ell'})$ is equal to a constant $C_1>0$.
		Since $\{\rho_\ell\}_\ell$ is a compatible system and $f(\chi(\rho_\ell(Frob_v)))$
		depends only on the characteristic polynomial of $\rho_\ell(Frob_v)$,
		the natural density $d(\mathscr S_\ell)$ is independent of $\ell$.
		\end{proof}
		
	Assume on the contrary there is an infinite set $\mathcal L$ of rational primes such that 
the full algebraic envelope
$\widehat\uG_\ell$ is connected (i.e., $\widehat\uG_\ell=\uG_\ell$)
for all $\ell\in\mathcal L$.
	For $\ell\in\mathcal L$, define $\overline{f\circ\chi}:\GL_{n,\F_\ell}\to \A_{\F_\ell}$
	as the mod $\ell$ reduction of the $\Z$-morphism $f\circ\chi$ and the $\F_\ell$-subvariety
	\begin{equation}\label{Zell}
	\uZ_\ell:=\uG_\ell\cap \{\overline{f\circ\chi}=0\}\subset\uG_\ell.
	\end{equation}
		
 \begin{prop}\label{bound}
There is a constant $C_2>0$ such that 
$|\uZ_{\ell}(\F_{\ell})|\le C_2 \ell^{\dim(\uG_{\ell})-1}$
for all sufficiently large $\ell\in\mathcal L$. 
\end{prop}

\begin{proof}
By Proposition \ref{denindep}, the intersection $\bar\Gamma_\ell\cap \uZ_\ell\neq \emptyset$.
Since $\overline{f\circ\chi}(\mathrm{id})\neq 0$ for $\ell\gg0$,
the subvariety $\uZ_\ell\subset\uG_\ell$ is a hypersurface  for $\ell\gg0$.
Since $\GL_{n,\F_\ell}$ is an open subscheme of the projective space $\mathbb{P}_{\F_\ell}^{n^2}$,
we denote by $\uZ_\ell'$ the Zariski closure of $\uZ_\ell$ in $\mathbb{P}_{\F_\ell}^{n^2}$.
The assertion holds by Lemma \ref{LW} (Lang-Weil) if we can show that 
the degree of $\uZ_\ell'$ (or $\uZ_{\ell,\overline\F_\ell}'$) 
is uniformly bounded independent of $\ell\in\mathcal L$.

By Proposition \ref{finite}, there is a finite set $\{G_i\subset\GL_{n,\Z[1/N]}:~1\leq i\leq m\}$
of connected split reductive subgroup subschemes  
such that $\uG_{\ell,\overline\F_\ell}$
is conjugate to some $G_{i,\overline\F_\ell}$ in $\GL_{n,\overline\F_\ell}$ 
for all $\ell\gg0$. It suffices to consider the case $m=1$. 
Since $\overline{f\circ\chi}$ is conjugation-invariant, the base change
$\uZ_{\ell,\overline\F_\ell}$ of \eqref{Zell} and 
\begin{equation}\label{new}
(G_{1}\cap \{f\circ\chi=0\})_{\overline\F_\ell}=G_{1,\overline\F_\ell}\cap \{\overline{f\circ\chi}=0\}_{\overline\F_\ell}
\end{equation}
are conjugate in $\GL_{n,\overline\F_\ell}$.
Since the degree of the Zariski closure of \eqref{new} in $\mathbb{P}_{\overline\F_\ell}^{n^2}$
is uniformly bounded 
independent of $\ell\gg0$, this is also true for the degree of $\uZ_{\ell,\overline\F_\ell}'$ for 
all $\ell\in\mathcal L$.
\end{proof}

We need two more estimates. For all $\ell\in\mathcal L$, we have $\bar\Gamma_\ell\subset\uG_\ell(\F_\ell)$.
Moreover, there is a constant $C_3>0$ such that 
 		\begin{equation}\label{5}
			|\uG_\ell(\F_\ell)|\leq C_3 |\bar\Gamma_\ell|\hspace{.2in}\forall \ell\in\mathcal L
		\end{equation}
		by Theorem \ref{Hui2}(ii). Define
$$\bar{\mathscr S_\ell}:=\{v\in\Sigma_K:~\bar\rho_\ell^{\ss}~\text{is unramified at }v\text{~and~} 
\overline{f\circ \chi}(\bar\rho_\ell^{\ss}(\mathrm{Frob}_v))=0\}.$$
Since  $\mathscr S_\ell\subset  \bar{\mathscr S_\ell}$, Proposition \ref{denindep} implies that
		\begin{equation}\label{4}
			\frac{|\uZ_{\ell}(\mathbb{F}_{\ell})\cap\bar\Gamma_{\ell}|}{|\bar\Gamma_{\ell}|}
			=d(\bar{\mathscr S_\ell})\ge d(\mathscr S_\ell)=C_1>0\hspace{.2in}\forall \ell\in\mathcal L.
		\end{equation}

By combining all the estimates, we obtain the following inequalities

		\begin{align*}
		\begin{split}
			C_2\ell^{\dim(\uG_\ell)-1}\stackrel{Prop. \ref{bound}}{\ge}|\uZ(\mathbb{F}_{\ell})|\ge|\uZ(\mathbb{F}_{\ell})\cap\bar\Gamma_{\ell}|
			\stackrel{\eqref{4}}{\ge} C_1|\bar\Gamma_{\ell}|\stackrel{\eqref{5}}{\ge} \frac{C_1}{C_3}|\uG_\ell(\mathbb{F}_{\ell})|\stackrel{Lem. \ref{gpbound}}{\ge} 
			\frac{C_1}{C_3}(\ell-1)^{\dim(\uG_\ell)}
		\end{split}
		\end{align*}
		for all sufficiently large $\ell\in\mathcal L$,
		which is absurd as $\mathcal L$ is infinite. 
		This finishes the proof of Proposition \ref{mainprop}.
	\end{proof}

\section*{Acknowledgments}
The authors thank Davide Lombardo for his interest and comments on the article.
We also appreciate the referee for their careful review and constructive feedback. Additionally, Hui received partial support from the University Grants Committee of Hong Kong, China (Grant no. 17314522), and was sponsored by a Humboldt Research Fellowship from the Alexander von Humboldt Foundation, Germany

\end{document}